\documentclass{amsart}

\usepackage{amssymb,amsmath,amsthm,amsfonts,amsopn,url,xcolor,hyperref,enumerate,mathtools,microtype,MnSymbol}
\usepackage{hyperref}
\usepackage{comment}
\usepackage{enumitem}
\usepackage{tikz-cd}
\usepackage{xcolor}
\usepackage[mathscr]{euscript}

\setlist[enumerate,1]{label={(\arabic*)}}

\usepackage[margin=1.5in]{geometry}

\newcommand{\fp}{\mathfrak{p}}

\newcommand{\fm}{\mathfrak{m}}

\newcommand{\bA}{\mathbb{A}}

\newcommand{\bF}{\mathbb{F}}

\newcommand{\bP}{\mathbb{P}}
\newcommand{\bQ}{\mathbb{Q}}

\newcommand{\cF}{\mathcal{F}}

\newcommand{\cJ}{\mathcal{J}}

\newcommand{\cO}{\mathcal{O}}

\newcommand{\sJ}{\mathscr{J}}

\renewcommand{\div}{\operatorname{div}}

\newcommand{\Gal}{\operatorname{Gal}}
\newcommand{\Hom}{\operatorname{Hom}}

\newcommand{\image}{\operatorname{Im}}

\newcommand\Ram{\operatorname{Ram}}

\newcommand{\sHom}{\operatorname{\mathscr{H}om}}

\newcommand\Spec{\operatorname{Spec}}

\newcommand{\Tr}{\operatorname{Tr}}

\newtheorem{thm}{Theorem}[section]
\newtheorem{prop}[thm]{Proposition}
\newtheorem{lemma}[thm]{Lemma}
\newtheorem{cor}[thm]{Corollary}

\theoremstyle{definition}
\newtheorem{definition}[thm]{Definition}

\newtheorem{remark}[thm]{Remark}

\title{Multiplier ideals and klt singularities via (derived) splittings}
\author{Peter M. McDonald}

\thanks{This work was  partially supported by NSF RTG DMS-1840190.}

\begin{document}
\maketitle

\begin{abstract}
Let $X$ be a normal, excellent, noetherian scheme over $\Spec\bQ$ with a dualizing complex. In this note, we find an alternate characterization of the multiplier ideal of $X$, as defined by de Fernex--Hacon, by considering maps $\pi_*\omega_Y\to\cO_X$ where $\pi:Y\to X$ ranges over all regular alterations. As a corollary to this result, we give a derived splinter characterization of klt singularities, akin to the characterization of rational singularities given by Kov\'acs and Bhatt. We also give an analogous description of the test ideal in characteristic $p>2$ as a corollary to a result of Epstein--Schwede.
\end{abstract}

\section{Introduction}
Let $X$ be a normal, excellent, noetherian scheme over $\Spec\bQ$ with a dualizing complex. For $\Delta\geq0$ a $\bQ$-divisor on $X$ such that $K_X+\Delta$ is $\bQ$-Cartier, we can consider the multiplier ideal $\sJ(X,\Delta)\subseteq\cO_X$ as a measure of the severity of the singularities of the pair $(X,\Delta)$. Introduced in the analytic setting by Nadel \cite{Nadel:1990} and further developed by Demailley, Siu and others \cite{Demailly:1993,Siu:1993,Angehrn-Siu:1995,Siu:1998}, Esnault and Viehweg \cite[Chapter~7]{Esnault-Viehweg:1992} independently developed the theory in the algebro-geometric setting. Lipman also encountered multiplier ideals in connections with his work on the Briancon--Skoda theorem \cite{Lipman:1994}, and their applications to algebra were studied by many others \cite{Ein-Lazarsfeld:1997,Kawamata:1999a,Kawamata:1999b,Ein-Lazarsfeld:1999,Demailly-Ein-Lazarsfeld:2000,Ein-Lazarsfeld-Smith:2001}.

While $\sJ(X,\Delta)$ depends on the geometry of the pair $(X,\Delta)$, de Fernex and Hacon introduced in \cite{dFH09} an object $\sJ(X)$ whose definition doesn't require the choice of a boundary divisor while also showing that $\sJ(X)$ is the unique maximal element in the collection $\{\sJ(X,\Delta)\}$ where $\Delta$ ranges over all effective $\bQ$-divisors on $X$ such that $K_X+\Delta$ is $\bQ$-Cartier. We call $\sJ(X)$ the multiplier ideal of $X$ and say that $X$ is has \emph{klt type} if $\sJ(X)=\cO_X$.

We set out to prove an alternate characterization of the multiplier ideal by considering maps $\pi_*\omega_Y\to\cO_X$ where $\pi:Y\to X$ is a regular alteration. This leads us to the main result of this paper:

\begin{thm}\label{main}
Let $X$ be a normal, excellent, noetherian scheme over $\Spec\bQ$ with a dualizing complex and $I=\prod\sJ_k^{a_k}$ be a formal effective $\bQ$-linear combination of ideal sheaves on $X$. Fix $L$ an algebraic closure of the function field of $X$. The multiplier ideal $\sJ(X,I)$ can be realized as
\[\sJ(X,I)=\sum_{\pi:Y\to X}\image\left(\sHom_X(\pi_*\omega_Y,\cO_X)\otimes_{\cO_X}\pi_*\omega_Y(-E_Y)\to\cO_X\right)\]
where $\pi\colon Y\to X$ ranges over all log regular alterations of $(X,I)$ whose function fields are contained in $L$ with $E_Y=\sum a_kE_k$ for $\sJ_k\cO_Y=\cO_Y(-E_k)$ and the map $\sHom_X(\pi_*\omega_Y,\cO_X)\otimes_{\cO_X}\pi_*\omega_Y(-E_Y)\to\cO_X$ is the evaluation map.
\end{thm}

Roughly speaking, we can interpret this as saying, at least locally, that an element $f\in\cJ(X,I)$ if and only if it is in the image of some map $\phi:\pi_*\omega_Y\to\cO_X$ restricted to $\pi_*\omega_Y(-E_Y)$. Showing that $\sJ(X,I)$ is contained in this sum of images is fairly straightforward. It is in showing the reverse containment where the work is done by proving the following key lemma inspired by Lemma 1.1 of \cite{FG12}, though the proof technique is quite different.

\begin{lemma}\label{fgpreview}
Let $\rho\colon \Spec S\to\Spec R$ be a finite map of normal, excellent, noetherian domains containing $\bQ$ with dualizing complexes and let $I=\prod\sJ_k^{a_k}$ be a formal effective $\bQ$-linear combination of ideal sheaves on $\Spec R$. Suppose we have a $\bQ$-Cartier divisor $\Gamma\geq \rho^*K_R$ with $n\Gamma=\div g$. Then 
\[\Tr_{S/R}\left(\rho_*\sJ(\omega_S,\Gamma,IS)\right)\subseteq\sJ(R,I).\]
In the case where $I=R$ and $\Tr_{S/R}\left(\rho_*\sJ(\omega_S,\Gamma)\right)=R$, this implies that $R$ has klt type.
\end{lemma}

The original motivation for this work was to develop a derived splinter characterization of klt singularities, akin to the following characterization of rational singularities by Bhatt and Kov\'acs \cite{Bha12, Kov00}:

\begin{thm}[\cite{Bha12} Theorem 2.12, \cite{Kov00} Theorem 3]
A scheme $X$ of essentially finite type over a field of characteristic $0$ has rational singularities if and only if it is a derived splinter, meaning for every proper surjective map $\pi:Y\to X$, the natural map $\cO_X\to R\pi_*\cO_Y$ splits in the derived category of coherent sheaves on $X$.
\end{thm}

If $X$ has klt type, then it has rational singularities and hence $\cO_X\to R\pi_*\cO_Y$ splits for any proper surjective map $\pi:Y\to X$. Thus, we might expect that some extra conditions on the splitting $R\pi_*\cO_Y\to\cO_X$ may characterize schemes having klt type. Ultimately, such a characterization follows as a corollary to Theorem \ref{main}:

\begin{cor}
Let $X$ be a normal, excellent, noetherian scheme over $\Spec\bQ$ with a dualizing complex and $I=\prod\sJ_k^{a_k}$ be a formal effective $\bQ$-linear combination of ideal sheaves on $X$. The following are equivalent
\begin{enumerate}
\item $(X,I)$ has klt type
\item For all sufficiently large regular alterations $\pi\colon Y\to X$, the natural map $\cO_X\to R\pi_*\cO_Y$ splits and locally factors through $R\pi_*\omega_Y(-E_Y)=\pi_*\omega_Y(-E_Y)$
\item For all sufficiently large regular alterations $\pi\colon Y\to X$, $\cO_X$ is locally a summand of $R\pi_*\omega_Y(-E_Y)=\pi_*\omega_Y(-E_Y)$
\end{enumerate}
Here, $E_Y:=\sum a_kE_k$ where $\cO_Y(-E_Y)=\sJ_k\cO_Y$. By sufficiently large regular alteration we mean any regular alteration $\pi\colon Y\to X$ such that
\[\sJ(X)=\image(\sHom(\pi_*\omega_Y,\cO_X)\otimes_{\cO_X}\pi_*\omega_Y(-E_Y)\to\cO_X).\]
\end{cor}

In positive characteristic, the test ideal plays a role analogous to the multiplier ideal and can be used to a define \emph{strongly $F$-regular} singularities, an analog of klt singularities in positive characteristic. In \cite{BST15}, Blickle, Schwede and Tucker give a characteristic-free description of an ideal $J(X,\Delta)$ that specializes to the multiplier ideal in characteristic zero and the test ideal in characteristic $p>0$, so we might expect that our characterization of the multiplier ideal might carry over to a characterization of the test ideal in positive characteristic. We achieve such a result in characteristic $p>2$ as a corollary to a result of Epstein and Schwede \cite{ES14} combined with the existence of quasi-Gorenstein finite covers.

\begin{prop}
Let $R$ be a Noetherian, $F$-finite reduced ring of characteristic $p>2$ and fix $L_1,\dots, L_n$, algebraic closures of $(R/\fp_i)_{\fp_i}$ for $\fp_i$ the minimal primes of $R$. The test ideal $\tau(R)$ can be realized as
\[\tau(R)=\sum_{R\subset S} \image\left(\Hom_R(\omega_S,R)\otimes_R\tau(\omega_S)\to R\right)\]
where the sum ranges over all finite extensions $R\subset S$ whose total ring of fraction is contained within $\prod L_i$ and $\Hom_R(\omega_S,R)\otimes_R\tau(\omega_S)\to R$ is the evaluation map.
\end{prop}

Here $\tau(\omega_S)\subseteq\omega_S$ is the parameter test submodule which plays an analogous role to $\pi_*\omega_Y$, which is sometimes called the Grauert--Riemenschneider sheaf or multiplier submodule, in characteristic zero. Note that the characteristic assumption is an artefact of the construction of quasi-Gorenstein covers in Lemma \ref{qgor}. See Remark \ref{char} for a brief discussion.

\begin{remark}
The above results could be stated for regular alterations by a similar argument outlined in the characteristic zero case for the multiplier ideal using the fact that $\tau(X)=\sum\tau(X,\Delta)$ for finitely many log-$\bQ$-Gorenstein pairs $(X,\Delta)$ (see \cite{Sch11}). However, we leave the result in this format as the local statement statement is more readily applicable to questions in local algebra.
\end{remark}

This also gives us the following splinter characterization of strongly $F$-regular singularities as a byproduct.

\begin{cor}
Let $R$ be a Noetherian, $F$-finite reduced ring of characteristic $p>2$. Then $R$ is strongly $F$-regular if and only if $R$ is a summand of $\tau(\omega_S)$ for any finite extension $R\subset S$ such that $\tau(R)=\image(\Hom_R(\omega_S,R)\otimes_R\tau(\omega_S)\to R)$.
\end{cor}
\subsection*{Acknowledgements}

The author would like to thank his advisor Karl Schwede for many helpful comments and conversations and his continued support in the writing of this paper, as well as S\'andor Kov\'acs, Linquan Ma, and Kevin Tucker for additional helpful conversations. He would also like to thank the referee for several comments improving the paper.

\section{Characteristic zero}

\subsection*{Preliminaries}

Throughout this section, all schemes are noetherian, normal and integral over a field of characteristic zero. We will often additionally require our schemes be excellent with a dualizing complex, but will always make this explicit. Before we discuss our main object of study, the multiplier ideal, we will review some preliminaries about canonical modules and the trace map. We assume that the reader is familiar with canonical modules at the level of \cite{Har77} and \cite{KM98}. This discussion largely follows Section 2 in \cite{BST15}, but we include it here for the convenience of the reader.

Given a normal integral scheme $X$ with canonical sheaf $\omega_X$, we say an integral divisor $K_X$ is a \emph{canonical divisor} for $X$ if $\cO_X(K_X)\cong\omega_X$. Given $\pi:Y\to X$ a proper generically finite map of normal integral schemes over a field $k$, we can consider the trace map
\[\Tr_{Y/X}:\pi_*\omega_Y\to \omega_X\]
Since any proper generically finite map can be factored as a proper birational map followed by a finite map, we discuss the trace map in these contexts.

Let $\pi:Y\to X$ be a proper birational map of normal integral schemes and fix a canonical divisor $K_Y$ on $Y$ and set $K_X=\pi_*K_Y$ (which ensures $K_Y$ and $K_X$ agree on the locus where $\pi$ is an isomorphism). Because $\pi$ is an isomorphism outside a codimension 2 subset of $X$, $\pi_*\cO_Y(K_Y)$ is a torsion-free sheaf whose reflexification is $\cO_X(K_X)$ and the trace map is simply the natural reflexification map $\pi_*\cO_Y(K_Y)\hookrightarrow\cO_X(K_X)$.

If $\pi:Y\to X$ is a finite surjective map of normal integral schemes, then $\pi_*\omega_Y\cong\sHom_{\cO_X}(\pi_*\cO_Y,\omega_X)$. We can then identify the trace map with the evaluation-at-1 map, $\sHom_{\cO_X}(\pi_*\cO_Y,\omega_X)\to\omega_X$. Assuming additionally that $\pi:Y\to X$ is a finite separable map of normal integral schemes with ramification divisor $\Ram_\pi$, we fix a canonical divisor $K_X$ on $X$ and set $K_Y=\pi^*K_X+\Ram_\pi$. Then the field-trace map
\[\Tr_{K(Y)/K(X)}:K(Y)\to K(X)\]
restricts to a map $\pi_*\cO_Y(K_Y)\to\cO_X(K_X)$ which can be identified with the Grothendieck trace map. Throughout the rest of the paper, whenever we have a proper generically finite map of normal integral schemes $\pi:Y\to X$, we will always choose $K_X$ and $K_Y$ compatibly according to the above discussion.

We are now ready to introduce the concept of pairs. A $\bQ$-divisor $\Gamma$ on $X$ is a formal linear combination of prime Weil divisors with coefficients in $\bQ$. Writing $\Gamma=\sum a_iZ_i$ where the $Z_i$ are distinct prime divisors, we use $\lceil\Gamma\rceil=\sum\lceil a_i\rceil Z_i$ and $\lfloor\Gamma\rfloor=\sum\lfloor a_i\rfloor Z_i$ to denote the round up and round down of $\Gamma$, respectively. We say that $\Gamma$ is $\bQ$-Cartier if there exists an integer $n>0$ such that $n\Gamma$ is an integral Cartier divisor, and the smallest such $n$ is called the \emph{index} of $\Gamma$.

\begin{definition}
A \emph{pair} $(X,\Delta)$ is the combined data of a normal integral scheme $X$ together with a $\bQ$-divisor $\Delta$ on $X$. The pair $(X,\Delta)$ is called \emph{log-$\bQ$-Gorenstein} if $K_X+\Delta$ is $\bQ$-Cartier.
\end{definition}

\begin{definition}
Given a log-$\bQ$-Gorenstein pair $(X,\Delta)$, a \emph{log resolution of singularities} of the pair $(X,\Delta)$ is a resolution of singularities $\pi:Y\to X$ such that $\text{except}(\pi)$ is a divisor and $\pi_*^{-1}(\Delta)+\text{except}(\pi)$ has simple normal crossing support.

More generally, let $I=\prod\sJ_k^{a_k}$ be an effective formal $\bQ$-linear combination of ideal sheaves on $X$. A \emph{log resolution} of $(X,I)$ is a proper birational morphism $\pi:Y\to X$ with $Y$ smooth such that for every $k$ the sheaf $\sJ_k\cO_Y\cong\cO_Y(-E_k)$ for $E_k\geq0$ $\bQ$-Cartier, $\text{except}(\pi)$ is also a divisor, and $\text{except}(\pi)+ E_Y$ has simple normal crossing support where $E_Y=\sum a_kE_k$. Log resolutions exist when $X$ is quasi-excellent by \cite[Theorem~2.3.6]{Temkin:2008}.
\end{definition}

For our purposes, we will also want to move beyond resolutions of singularities to consider regular alterations.

\begin{definition}
A map $\pi:Y\to X$ of schemes is a regular alteration if it is surjective, proper, and generically finite and $Y$ is nonsingular.

The main benefit to this is the Stein factorization. If $\pi:Y\to X$ is a regular alteration, then $\pi$ factors as
\[\begin{tikzcd}
Y\arrow[rd,"\tau"]\arrow[rr,"\pi"]& &X\\
&Z\arrow[ru,"\rho"]&
\end{tikzcd}\]
where $\rho:Z\to X$ is a finite surjective map and $\tau:Y\to Z$ is a resolution of singularities.

Given a log $\bQ$-Gorenstein pair $(X,\Delta)$, we will say that $\pi:Y\to X$ is a log regular alteration if it is a regular alteration, $\text{except}(\tau)$ is a divisor, and $\pi_*^{-1}(\Delta)+\text{except}(\tau)$ has simple normal crossing support.
\end{definition}

This then leads us to the object of interest: the multiplier ideal. The theory of multiplier ideals was largely developed by Esnault and Viehweg in our setting \cite{Esnault-Viehweg:1992}, and more details on the theory can be found in \cite{Laz04}.

\begin{definition}
Let $X$ be a normal, excellent, noetherian scheme over $\Spec\bQ$ with a dualizing complex. Given a log-$\bQ$-Gorenstein pair $(X,\Delta)$ with $\Delta\geq0$, the \emph{multiplier ideal} of the pair $(X,\Delta)$ is
\[\sJ(X,\Delta):=\pi_*\cO_Y(\lceil K_Y-\pi^*(K_X+\Delta)\rceil)\subseteq\cO_X\]
where $\pi:Y\to X$ is a log resolution of singularities of the pair $(X,\Delta)$. Note that \cite[Theorem~A]{Mur21} ensures that $R\pi_*\cO_Y(\lceil K_Y-\pi^*(K_X+\Delta)\rceil)=\pi_*\cO_Y(\lceil K_Y-\pi^*(K_X+\Delta)\rceil)$.
\end{definition}

More generally, we can let $\pi:Y\to X$ be a log regular alteration of $(X,\Delta)$ and consider \[\sJ(X,\Delta)=\Tr_{Y/X}(\pi_*\cO_Y(\lceil K_Y-\pi^*(K_X+\Delta)\rceil))\subseteq\cO_X\]
via Theorem 8.1 of \cite{BST15}.

We will also use the related concept of \emph{multiplier submodules}, sometimes called Grauert--Riemenschneider sheaves as they were first suggested as objects of study by Grauert and Riemenschneider \cite{Grauert-Riemenschneider:1970}. As far as we know, the first instance of the name multiplier submodule appears in \cite{Bli04}.

\begin{definition}[\cite{Bli04}]\label{multsubmod}
Let $X$ be a normal, excellent, noetherian scheme over $\Spec\bQ$ with a dualizing complex $\omega_X^\bullet$. Let $\omega_X$ be the canonical module, the last non-vanishing cohomology sheaf of $\omega_X^\bullet$, and fix $K_X$ a canonical divisor with $\cO_X(K_X)\cong\omega_X$. Given an effective $\bQ$-Cartier divisor $\Delta$, we define the multiplier submodule $\sJ(\omega_X,\Delta)$ as 
\[\sJ(\omega_X,\Delta):=\pi_*\cO_Y\left(\left\lceil K_Y-\pi^*\Delta\right\rceil\right)\]
where $\pi:Y\to X$ is a log resolution of $(X,\Delta)$. Once again, \cite[Theorem~A]{Mur21} ensures $R\pi_*\cO_Y\left(\left\lceil K_Y-\pi^*\Delta\right\rceil\right)=\pi_*\cO_Y\left(\left\lceil K_Y-\pi^*\Delta\right\rceil\right)$.
\end{definition}

The hypothesis that $K_X+\Delta$ be $\bQ$-Cartier is included because there is a well-defined theory of pullbacks for Cartier divisors. Work of de Fernex and Hacon \cite{dFH09} removes this hypothesis by defining a pullback operation that uses the fractional ideal sheaf corresponding to a divisor.

Assume the same hypotheses as \ref{multsubmod}. Given a divisor $D$ on $X$ and $\pi:Y\to X$ proper, birational, de Fernex and Hacon define the natural pullback of $D$ along $\pi$ to be 
\[\pi^\natural(D):=\div_Y(\cO_X(-D)\cdot\cO_Y).\]
In particular, $\cO_Y(-\pi^\natural D)=(\cO_X(-D)\cdot\cO_Y)^{\vee\vee}$ where $\cF^\vee:=\sHom(\cF,\cO_Y)$ for any quasi-coherent sheaf on $Y$. Applying this operation to multiples of $K_X$, we come to the following definition.

\begin{definition}[\cite{dFH09} 2.6, 3.1]
Given a proper birational morphism of normal, excellent, noetherian schemes over $\Spec\bQ$ with dualizing complexes $\pi:Y\to X$, the $m^{th}$ limiting relative canonical divisor $K_{m,Y/X}$ is 
\[K_{m,Y/X}:=K_Y-\frac{1}{m}\pi^\natural(mK_X)\]
If $I=\prod\sJ_k^{a_k}$ be an effective formal $\bQ$-linear combination of ideal sheaves on $X$, define
\[\sJ_m(X,I):=\pi_*\cO_{Y}(\lceil K_{m,Y/X}-E_Y\rceil)\]
where $\pi:Y\to X$ is a log resolution of $(X,I+\cO_X(-mK_X))$.
\end{definition}

\begin{prop}[cf. \cite{dFH09} Proposition 4.7]\label{max}
    Let $X$ be a normal, excellent, noetherian scheme over $\Spec\bQ$ with a dualizing complex and $I=\prod\sJ_k^{a_k}$ be a formal effective $\bQ$-linear combination of ideal sheaves on $X$. The collection $\{\sJ_m(X,I)\}_{m\geq1}$ has a unique maximal element.
\end{prop}
\begin{proof}
    Fix $\pi:Y\to X$ a proper birational morphism of normal, excellent, noetherian schemes over $\Spec\bQ$ with dualizing complexes. We first note that $K_{m,Y/X}\leq K_{mq,Y/X}$ for all $m,q$ positive integers as $mv^\natural(D)\geq v^\natural(mD)$. Thus $\sJ_m(X,\Delta)\subseteq\sJ_{mq}(X,\Delta)$ and the existence of a unique maximal element follows by noetherianity.
\end{proof}

\begin{definition}
    Let $X$ be a normal, excellent, noetherian scheme over $\Spec\bQ$ with a dualizing complex and $I=\prod\sJ_k^{a_k}$ be a formal effective $\bQ$-linear combination of ideal sheaves on $X$. We call the unique maximal element of $\{\sJ_m(X,I)\}_{m\geq1}$ the \emph{multiplier ideal of the pair} $(X,I)$ and denote it $\sJ(X,I)$. When $I=\cO_X$, we call this unique maximal element the \emph{multiplier ideal} of $X$ and denote it by $\sJ(X)$. If $\sJ(X)=\cO_X$, we say that $X$ has \emph{klt type}.
\end{definition}

\begin{remark}\label{rmk}
For sufficiently divisible $n$, $\sJ(X)=\sJ(\omega_X,(\omega_X^{(-n)})^{1/n})$. To see this, let $\pi:Y\to X$ be a log resolution of $(X,\cO_X(-nK_X))$ and note that
\begin{align*}
\sJ(X)&=\pi_*\cO_Y\left(\left\lceil K_Y-\frac{1}{n}\pi^\natural(nK_X))\right\rceil\right)\\
&=\pi_*\cO_Y\left(\left\lceil K_Y-\frac{1}{n}\div_Y\left(\cO_X(-nK_X)\cdot\cO_Y\right)\right\rceil\right)\\
&=\sJ(\omega_X,(\omega_X^{(-n)})^{1/n}).
\end{align*}
\end{remark}

If we additionally assumed that $X$ is a variety, \cite[Corollary~5.5]{dFH09} shows that $\sJ(X,I)$ is also the unique maximal element of $\{\sJ((X,\Delta);I)\}$. This follows from \cite[Proposition~5.2]{dFH09}, which says that $\sJ(X,I)=\sJ_m(X,I)=\sJ((X,\Delta);I)$ whenever $\Delta$ is what is called \emph{$m$-compatible}, and \cite[Theorem~5.4]{dFH09}, which shows that an $m$-compatible $\Delta$ exists for all $m\geq0$. We will introduce a variant of their result in our more general setting, which will require defining an $m$-compatible boundary and a Bertini theorem from \cite{LM22}.

\begin{definition}[\cite{dFH09} Definition 5.1]
    Let $X$ be a normal, excellent, noetherian scheme over $\Spec\bQ$ with a dualizing complex and $I=\prod\sJ_k^{a_k}$ be a formal effective $\bQ$-linear combination of ideal sheaves on $X$. Fix an integer $m\geq2$. Given a log resolution $\pi:Y\to X$ of $(X,I\cO_X(-mK_X))$, we say that the log $\bQ$-Gorenstein pair $(X,\Delta)$ is $m$-compatible for $(X,I)$ with respect to $\pi$ if
    \begin{enumerate}
        \item[(i)] $m\Delta$ is integral and $\lfloor\Delta\rfloor=0$,
        \item[(ii)] no component of $\Delta$ is contained in $Z_k$ for any $k$; where $Z_k$ is the subscheme defined by $\sJ_k$
        \item[(iii)] $\pi$ is a log resolution of the pair $((X,\Delta);I\cO_X(-mK_X))$
        \item[(iv)] $K_Y+\Delta_Y-\pi^*(K_X+\Delta)=K_{m,Y/X}$ where $\Delta_Y$ is the proper transform of $\Delta$ on $Y$
    \end{enumerate}
\end{definition}

\begin{thm}[Theorem 10.1 \cite{LM22}]\label{bertini1}
    Let $(R,\fm,k)$ be a Noetherian local domain containing $\bQ$. Fix an integer $N\geq1$. Let $f:X\to\bP^N_R$ be a separated morphism of finite type from a regular Noetherian scheme $X$. Assume that every closed point of $X$ lies over the unique closed point of $\Spec R$.

    Let $T_0,T_1,\dots,T_N$ be a basis of $H^0(\bP^N_R,\cO(1))$ as a free $R$-module. Then, there exists a nonempty Zariski open subset $W\subseteq\bA^{N+1}_k$ with the following property: For all $a_0,a_1,\dots,a_N\in R$, if
    \[(\bar{a}_0,\bar{a}_1,\dots,\bar{a}_N)\in W(k),\]
    then the section
    \[h=a_0T_1+a_1T_1+\cdots+a_NT_N\]
    is such that $f^{-1}(V(h))$ is regular.
\end{thm}

We will want a slightly more specific application of this result for our purposes, so we include that statement here, making no claims over its ownership.

\begin{lemma}\label{bertini2}
    Let $(R,\fm,k)$ be a Noetherian local domain containing $\bQ$. Let $J=(g_0,\dots,g_r)$ be an ideal of $R$, let $I=\prod\sJ_k^{a_k}$ be an effective $\bQ$-linear combination of ideal sheaves on $X=\Spec R$, and let $\pi:Y\to X$ be a log resolution of $(X,I,J)$. Then for a general choice of $(a_0,\dots,a_r)\in\bA_k^{n+1}$ we have that, for $f\colon=a_0g_0+\cdots +a_rg_r$, $\div_X f$ is reduced and avoids the components of $I$ and $\pi:Y\to X$ is a log resolution of $(X,I,\div_X f)$.
\end{lemma}
\begin{proof}
Let $\tilde{X}$ be the blowup of $J$ in $R$. Consider the following diagram coming from the universal property of the blowup
    \[
    \begin{tikzcd}
        Y\arrow[r,"\tau"]\arrow[drr,"\pi"]&\tilde{X}\arrow[r,"i"]\arrow[dr,"b"]&\bP_R^r\arrow[d,"p"]\\
        &&X
    \end{tikzcd}
    \]
where $i$ comes from the surjection of graded rings $R[T_0,\dots,T_r]\to Bl_I(R)$ sending $T_i\mapsto g_i$ in degree one. Let $h=a_0T_0+\cdots+a_rT_r$. Then $\div_{\tilde{X}}f=\div_{\tilde{X}}h$. By Theorem \ref{bertini1}, $\div_Yf$ is smooth for a Zariski dense subset of $(a_0,\dots,a_r)\in\bA^{r+1}_k$, and by \cite[Remark 10.2]{LM22}, we can also ensure that $\text{except}(\pi)+E_Y+\div_Yf$ is simple normal crossing (as $\pi:Y\to X$ is already a log resolution of $(R,I)$) and that $\div_Xf$ avoids the components of $I$.
\end{proof}

We now show that, working sufficiently locally, $m$-compatible boundaries exist in our setting and thus the multiplier ideal agrees with the multiplier ideal of some pair.

\begin{prop}[cf. \cite{dFH09} Theorem 5.4, Corollary 5.5]\label{mult}
    Let $R$ be a local, normal, excellent, noetherian domain containing $\bQ$ with a dualizing complex, let $X=\Spec R$ and $I=\prod\sJ_k^{a_k}$ be a formal effective $\bQ$-linear combination of ideal sheaves on $X$. Then
    \[\sJ(X,I)=\sJ((X,\Delta);I)\]
    for some log $\bQ$-Gorenstein pair $(X,\Delta)$.
\end{prop}
\begin{proof}
\pushQED{\qed}
    This proof is essentially identical to the work of de Fernex--Hacon but we include it here for the convenience of the reader. We first claim that $\sJ((X,I);\Delta)\subseteq\sJ(X,I)$ for any effective log $\bQ$-Gorenstein pair $(X,\Delta)$ \cite[cf.~Remark~5.3]{dFH09}. Suppose $m$ is the $\bQ$-Cartier index for $K_X+\Delta$. It is straightforward to show that $\pi^\natural(C+D)=\pi^\natural(C)+\pi^\natural(D)$ for $C$ Cartier, so $\pi^\natural(mK_X)=\pi^\natural(m(K_X+\Delta)-m\Delta)=m\pi^*(K_X+\Delta)+\pi^\natural(-m\Delta)$. This implies $K_Y+\Delta_Y-\pi^*(K_X+\Delta)\leq K_{m,Y/X}=K_Y-\frac{1}{m}\pi^\natural(mK_X)$ and thus that $\sJ((X,\Delta);I)\subseteq\sJ_m(X,I)\subseteq\sJ(X,I)$.
    
    We next claim that $\sJ((X,I);\Delta)=\sJ_m(X,I)$ for any $m\geq2$ and any $m$-compatible boundary $\Delta$ \cite[cf.~Proposition~5.2]{dFH09}. This is because, since $\Delta$ shares no common components with $Z$ and $\lfloor\Delta\rfloor=0$ we have, for any log resolution $\pi:Y\to X$ of $((X,\Delta);I)$,
    \begin{align*}
        \sJ((X,I);\Delta)&=\pi_*\cO_Y(\lceil K_Y-\pi^*(K_X+\Delta)-E_Y\rceil)\\
        &=\pi_*\cO_Y(\lceil K_Y+\Delta_Y-\pi^*(K_X+\Delta)-E_Y\rceil)\\
        &=\sJ_m(X,I).
    \end{align*}
    
Finally, we claim that we can find an $m$-compatible log $\bQ$-Gorenstein pair $(X,\Delta)$ for every $m\geq2$ \cite[cf.~Theorem~5.4]{dFH09}. Let $D$ be an effective divisor such that $K_X-D$ is $\bQ$-Cartier and let $\pi:Y\to X$ be a log resolution of $(X,\cO_X(-mK_X)+\cO_X(-mD))$ and let $E=\pi^\natural(mD)$. By Lemma \ref{bertini2}, we can find $g\in\cO_X(-mD)$ such that $\div_Xg\colon=G=M+mD$ is reduced and shares no common components with $D$ or $I$. Set $\Delta\colon=\frac{1}{m}M$. Note that by design, $K_X+\Delta=K_X-D+\frac{1}{m}G$ is $\bQ$-Cartier and $\frac{1}{m}\pi^*G=\Delta_Y+\frac{1}{m}E$. To se that $(X,\Delta)$ is an $m$-compatible log $\bQ$-Gorenstein pair, note that $m\Delta$ is integral and $\lfloor\Delta\rfloor=0$. Working generally enough, we can assume $\pi$ is also a log resolution of $((X,\Delta);I)$. Finally, note that
    \begin{align*}
        K_Y+\Delta_Y-\pi^*(K_X+\Delta)&=K_Y+\Delta_Y-\pi^*(K_X+\Delta-\frac{1}{m}G)-\frac{1}{m}\pi^*G\\
        &=K_Y-\pi^*(K_X-D)-\frac{1}{m}E\\
        &=K_{m,Y/X}. \qedhere
    \end{align*}
\end{proof}

The last result we will need is a variant of \cite[Proposition~9.2.28]{Laz04}.

\begin{prop}\label{pos}
    Let $X=\Spec R$ be a local, normal, excellent, noetherian domain containing $\bQ$ with a dualizing complex and let $I=\prod\cJ_k^{a_k}$ an effective formal $\bQ$-linear combination of ideal sheaves on $\Spec R$. Let $J\subseteq R$ be an ideal and fix $c>0$ a rational number. For $k>c$, we can find find $f_1,\dots,f_k\in J$ such that, for $D=\frac{1}{k}\sum_i\div_Xf_i$, 
    \[\sJ(\omega_R,I,J^c)=\sJ(\omega_R,I,c\cdot D).\]
\end{prop}
\begin{proof}
\pushQED{\qed}
Let $\pi:Y\to X$ be a log resolution of $(X,I,J^c)$ and let $J\cO_Y=\cO_Y(-B)$ for $B\geq0$ $\bQ$-Cartier. By repeated applications of Lemma \ref{bertini2}, we can choose $f_1,\dots,f_k$ generally so that $\pi:Y\to X$ is a log resolution of $(X,I,c\cdot D)$.

Write $\div_Yf_i=B+A_i$ for $A_i$ effective and reduced. Then
        \begin{align*}
            \sJ(\omega_R,I,c\cdot D)&=\pi_*\cO_Y\left(\left\lceil K_Y-E_Y-c\pi^*D\right\rceil\right)\\
            &=\pi_*\cO_Y\left(\left\lceil K_Y-E_Y-c\cdot B-\frac{c}{k}\sum A_i\right\rceil\right)\\
            &=\pi_*\cO_Y(\lceil K_Y-E_Y-c\cdot B\rceil)\\
            &=\sJ(\omega_R,I,J^c).\qedhere
        \end{align*}
\end{proof}

\subsection*{Results}

Throughout this section, fix $X$, a normal, excellent, noetherian scheme over $\Spec\bQ$ with a dualizing complex and $I=\prod\sJ_k^{a_k}$ be an effective formal $\bQ$-linear combination of ideal sheaves on $X$. Whenever we have a log resolution of singularities or log regular alteration of the pair $(X,I)$, we will use the notation $E_Y=\sum a_kE_k$ where $\sJ_k\cO_Y=\cO_Y(-E_k)$ is the ideal sheaf of an effective Cartier divisor $E_k$ on $Y$ for each $k$. We begin by showing one direction of the containment.

\begin{prop}\label{cont1}
Let $X$ be a normal, excellent, noetherian scheme over $\Spec\bQ$ with a dualizing complex. Then the multiplier ideal $\sJ(X,I)$ satisfies
\[\sJ(X,I)\subseteq\sum_{\pi:Y\to X}\image\left(\sHom_X(\pi_*\omega_Y,\cO_X)\otimes_{\cO_X}\pi_*\omega_Y(-E_Y)\to\cO_X\right)\]
where $\pi:Y\to X$ ranges over all regular log alterations of $(X,I)$ and the map \[\sHom_X(\pi_*\omega_Y,\cO_X)\otimes_{\cO_X}\pi_*\omega_Y(-E_Y)\to\cO_X\] is the evaluation map.
\end{prop}
\begin{proof}
We check containment locally. By Proposition \ref{mult}, let $\sJ(X,I)=\sJ((X,\Delta);I)$ for some $\Delta\geq0$ such that $K_X+\Delta$ is $\bQ$-Cartier. Let $\tau:\tilde{X}\to X$ be a log resolution of the triple $(X,\Delta,I)$. By \cite[Lemma~4.5]{BST15} we can finite a finite cover $\rho:Y\to\tilde{X}$ with $\pi=\tau\circ\rho$ such that $\pi*(K_X+\Delta)$ and $\rho^*E_{\tilde{X}}=E_Y$ are both effective Cartier divisors. Let $\pi^*(K_X+\Delta)=\div(g)\in\cO_Y$. Using the projection formula we get that
\begin{align*}
\Tr(g\cdot\pi_*\omega_Y(-E_Y))&=\Tr(\pi_*\cO_Y(K_Y-\div g-E_Y))\\
&=\Tr(\pi_*\cO_Y(K_Y-\pi^*(K_X+\Delta)-E_Y))\\
&=\sJ((X,\Delta);I)\\
&=\sJ(X,I)
\end{align*}
with the penultimate equality coming from Theorem 8.1 of \cite{BST15}.
\end{proof}

For the reverse containment, multiplier submodules on finite covers will play a key role. The key idea is that for $\pi:Y\to X$ a regular alteration, any map $\phi:\pi_*\omega_Y\to\cO_X$ factors through the multiplier submodule $\sJ(\omega_Z,\Gamma)$ of some divisor $\Gamma$ on $X$, where $\rho:Z\to X$ is the finite part of the Stein factorization of $\pi$. This ultimately implies, thanks to a key lemma, $\image\phi\subseteq\sJ(X)$ (as well as the analogous statement for pairs $(X,I)$). Before we proceed, we need the following fact about how multiplier submodules transform when we enlarge finite covers of our base.

\begin{lemma}\label{finmult}
Let $\Spec S'\xrightarrow{\phi} \Spec S\xrightarrow{\psi}\Spec R$ be finite surjective maps of normal, excellent, noetherian affine schemes over $\Spec\bQ$ with dualizing complexes with $\Gamma\geq0$ a $\bQ$-Cartier $\bQ$-divisor on $\Spec S$ and $I$ an effective formal $\bQ$-linear combination of ideal sheaves on $\Spec S$. Then
 \[\image\Tr_{S/R}\left(\psi_*\sJ(\omega_S,\Gamma,I)\right)=\image\Tr_{S'/R}\left(\psi_*\phi_*\sJ(\omega_{S'},\phi^*\Gamma,IS')\right).\]
\end{lemma}
\begin{proof}
\pushQED{\qed}
Use Proposition \ref{pos} to replace $I$ with some divisor $D$ and $IS'$ with $\phi^*D$. Because any regular alteration of $\Spec S'$ will also be a regular alteration of $\Spec S$, \cite[Theorem~8.1]{BST15} gives
\[\Tr_{S'/S}\left(\phi_*\sJ(\omega_{S'},\phi^*\Gamma,IS')\right)=\sJ(\omega_S,\Gamma,I).\]
\cite[Lemma~2.3]{BST15} tells us that $\Tr_{S'/R}=\Tr_{S/R}\circ\psi_*\Tr_{S'/S}$ and so we get that 
\begin{align*}
\image\Tr_{S/R}\left(\psi_*\sJ(\omega_S,\Gamma,I)\right)&=\image\Tr_{S/R}\circ\psi_*\Tr_{S'/S}\left(\phi_*\sJ(\omega_{S'},\phi^*\Gamma,IS')\right)\\
&=\image\Tr_{S'/R}\left(\psi_*\phi_*\sJ(\omega_{S'},\phi^*\Gamma,IS')\right). \qedhere
\end{align*}
\end{proof}

\begin{lemma}\label{fg}
Let $\rho:\Spec S\to\Spec R$ be a finite surjective map of normal, excellent, noetherian affine schemes over $\Spec\bQ$ with dualizing complexes and let $I=\prod\sJ_k^{a_k}$ be a formal effective $\bQ$-linear combination of ideal sheaves on $\Spec R$. Suppose we have a $\bQ$-Cartier divisor $\Gamma\geq \rho^*K_R$ with $n\Gamma=\div g$. Then 
\[\Tr_{S/R}\left(\rho_*\sJ(\omega_S,\Gamma,IS)\right)\subseteq\sJ(R,I).\]
In the case where $I=R$ and $\Tr_{S/R}\left(\rho_*\sJ(\omega_S,\Gamma)\right)=R$, this implies that $R$ has klt type.
\end{lemma}

\begin{proof}
Without lost of generality, we can assume $R\to S$ is a map of local rings. Fix $K_R\geq0$ and note that we can assume $R\to S$ is generically Galois. Indeed, let $S'$ be a the normalization of $S$ inside a Galois closure of $K(S)$ over $K(R)$. This a finite extension, so by replacing $\Gamma$ with its pullback along this extension and invoking Lemma \ref{finmult}, we can assume $R\to S$ is generically Galois.

Because $\div g\geq n\rho^*K_R$ we know that $g\in S(-n\rho^*K_R)$. Choosing a general element $f\in S(-n\rho^*K_R)$, Proposition \ref{pos} implies
\begin{align*}
\sJ(\omega_S,g^{1/n},IS)\subseteq \sJ(\omega_S,S(-n\rho^*K_R)^{1/n},IS)=\sJ(\omega_S,f^{1/n},IS).
\end{align*}
Let $G=\Gal(K(S)/K(R))$ and let
\[h=\prod_{\sigma\in G}\sigma(f)\]
Then we claim that
\[\sJ(\omega_S,h^{1/(n|G|)},IS)=\sJ(\omega_S,f^{1/n},IS)\]
To see this, let $Y$ be a $G$-equivariant log resolution of singularities of $(\Spec S,S(-n\rho^*K_R),I)$ (see \cite[Remark~2.1.5(ii)]{Temkin:2023}). Then for any $\sigma\in G$ we have the following commutative square
\[
\begin{tikzcd}
R(-nK_R)S\arrow[r]\arrow[d,"\sigma"]&S(-n\rho^*K_R)\arrow[d,"\sigma"]\\
R(-nK_R)S\arrow[r]&S(-n\rho^*K_R)
\end{tikzcd}
\]
The horizontal maps are given by reflexification and are thus functorial and so since the lefthand vertical map is an isomorphism, so is the righthand vertical map. Thus, the ideal $S(-n\rho^*K_R)$ is stable under the Galois action. We note that for $S(-n\rho^*K_R)\cdot\cO_Y=\cO_Y(-F)$, this implies that for any $\sigma\in G$,

\[\left\lfloor \frac{1}{n}F\right\rfloor=\left\lfloor\frac{1}{n}\sigma F\right\rfloor=\left\lfloor\frac{1}{n}\div_Yf\right\rfloor\]

and so the portion of $\frac{1}{n}\div_Yf$ that is not fixed by the Galois action will have coefficients less than $1$. This implies that
\[\left\lfloor\frac{1}{n|G|}\div_Yh\right\rfloor=\left\lfloor\frac{1}{n}\div_Yf\right\rfloor,\] 
proving the claim that
\[\sJ(\omega_S,h^{1/(n|G|)},IS)=\sJ(\omega_S,f^{1/n},IS).\]
Then because $h$ is Galois invariant and thus in $R(-n|G|K_R)$, an application of Lemma \ref{finmult} gives
\begin{center}
\begin{tikzcd}
\rho_*\sJ(\omega_S,h^{1/n|G|},IS)\arrow[r,"\Tr_{S/R}"]&\sJ(\omega_R,h^{1/n|G|},I)\subseteq\sJ(\omega_R,R(-n|G|K_R)^{1/n|G|},I)\subseteq R
\end{tikzcd}
\end{center}
and thus
\[\Tr_{S/R}(\rho_*\sJ(\omega_S,g^{1/n},IS)\subseteq\sJ(\omega_R,R(-n|G|K_R)^{1/n|G|},I).\]
By Proposition \ref{max} and the chain of equalities in Remark \ref{rmk},
\[\sJ(\omega_R,R(-n|G|K_R)^{1/n|G|},I)\subseteq\sJ(R,I),\]
with equality if we choose $n$ such that $n|G|$ is sufficiently divisible.
\end{proof}

Our goal now is to show that given a $\pi:Y\to X$ a log regular alteration of the pair $(X,I)$ and $\phi:\pi_*\omega_Y(-E_Y)\to\cO_Y$, we can find a divisor on a finite cover of $X$ such that the trace of the multiplier submodule corresponding to that divisor agrees with the image of $\phi$.

\begin{lemma}\label{cover}
Let $X=\Spec R$ be a normal, excellent, noetherian scheme over $\Spec\bQ$ with a dualizing complex and consider $\pi:Y\to X$ a log regular alteration of the pair $(X,I)$ and $\phi:\pi_*\omega_Y\to\cO_X$. Let $\rho:Z\to X$ be the finite portion of the Stein factorization. Then we can find a principal divisor $\Gamma=\div(h)\geq\rho^*K_X$ on $Z$ such that 
\[\image\left(\Tr:\rho_*\sJ(\omega_Z,\Gamma,I\cO_Z)\to \cO_X\right)=\image\left(\phi:\pi_*\omega_Y(-E_Y)\to\cO_X\right).\]
\end{lemma}

\begin{proof}
\pushQED{\qed}
Fix $K_X\geq0$. Because $X$ is $S_2$, $\phi$ factors through $(\pi_*\omega_Y)^{S_2}\cong\rho_*\omega_Z$, so by abuse of notation, we consider $\phi:\rho_*\omega_Z\to\cO_X$. By a standard Grothendieck duality argument we see
\begin{align*}
\rho_*\cO_Z(-\rho^*K_X)&\simeq\rho_*\sHom_Z(\omega_Z\otimes_Z\rho^*\omega_X,\omega_Z)\\
&\simeq\sHom_X(\rho_*\omega_Z\otimes_X\omega_X,\omega_X)\\
&\simeq\sHom_X(\rho_*\omega_Z,\cO_X)
\end{align*}
Taking global sections, we get that $\Gamma(Z,\cO_Z(-\rho^*K_X))\simeq\Hom_X(\rho_*\omega_Z,\cO_X)$. Tracing through the sequence of isomorphisms, we claim that this isomorphism sends $s\in\Gamma(Z,\cO_Z(-\rho^*K_X))$ to $\Tr_{Z/X}(s-)$. To see this, note that, at the level of global sections, the first isomorphism above sends $s$ to the multiplication-by-$s$ map $\mu_s:\omega_Z\otimes_Z\rho^*\omega_X\to\omega_Z$. The second isomorphism sends $\mu_s$ to $\Tr_{Z/X}\circ\rho_*\mu_s=\Tr_{Z/X}(s-)$ via Grothendieck duality. The final isomorphism can be understood by viewing $\Tr_{Z/X}(s-):K(Z)\to K(X)$ and thinking of the isomorphism as restriction to different $\cO_X$-submodules of $K(Z)$.

Then, let $\phi=\Tr_{Z/X}(h-)$ for $h\in\Gamma(Z,\cO_Z(-\rho^*K_X))$. Viewing $h$ as an element of $\cO_Z$, we get a divisor $\div h$ such that $\div h\geq\rho^*K_X$ and, by the projection formula, such that
\[\phi(\pi_*\omega_Y(-E_Y))=\Tr(h\cdot\pi_*\omega_Y(-E_Y))=\Tr(\rho_*\sJ(\omega_Z,\div h,I\cO_Z)). \qedhere\]
\end{proof}

We are now ready to show the reverse containment holds.

\begin{prop}\label{cont2}
With notation as above
\[\sJ(X)\supseteq\sum_{\pi:Y\to X}\image\left(\sHom_X(\pi_*\omega_Y,\cO_X)\otimes_{\cO_X}\pi_*\omega_Y(-E_Y)\to\cO_X\right)\]
\end{prop}

\begin{proof}
Since both these objects are $\cO_X$-submodules, we can check containment locally. By Lemma \ref{cover}, given $\pi:Y\to X$ a log regular alteration of the pair $(X,I)$ and $\phi:\pi_*\omega_Y\to\cO_X$, we can find a finite cover $Z\to X$ and a divisor $\Gamma$ on $Z$ such that $\Tr(\sJ(\omega_Z,\Gamma,I))=\phi(\pi_*\omega_Y(-E_Y))$. Then by Proposition \ref{fg} we can find $\Delta$ on $X$ such that $\image\phi\subseteq\sJ(X,\Delta,I)\subseteq\sJ(X,I)$.
\end{proof}

Together, Lemmas \ref{cont1} and \ref{cont2} give us our main result.

\begin{thm}[Theorem \ref{main}]
Let $X$ be a normal, excellent, noetherian scheme over $\Spec\bQ$ with a dualizing complex and $I=\prod\sJ_k^{a_k}$ be a formal effective $\bQ$-linear combination of ideal sheaves on $X$. The multiplier ideal $\sJ(X,I)$ in the sense of \cite{dFH09} can be realized as
\[\sJ(X,I)=\sum_{\pi:Y\to X}\image\left(\sHom_X(\pi_*\omega_Y,\cO_X)\otimes_{\cO_X}\pi_*\omega_Y(-E_Y)\to\cO_X\right)\]
where $\pi:Y\to X$ ranges over all log regular alterations of $(X,I)$ with $E_Y=\sum a_kE_k$ for $\sJ_k\cO_Y=\cO_Y(-E_k)$ and the map $\sHom_X(\pi_*\omega_Y,\cO_X)\otimes_{\cO_X}\pi_*\omega_Y(-E_Y)\to\cO_X$ is the evaluation map.
\end{thm}

As a corollary, we deduce the following derived splinter characterization of klt singularities:

\begin{cor}
Let $X$ be a normal, excellent, noetherian scheme over $\Spec\bQ$ with a dualizing complex and $I=\prod\sJ_k^{a_k}$ be a formal effective $\bQ$-linear combination of ideal sheaves on $X$. The following are equivalent
\begin{enumerate}
\item $(X,I)$ has klt type
\item For all sufficiently large regular alterations $\pi\colon Y\to X$, the natural map $\cO_X\to R\pi_*\cO_Y$ splits and locally factors through $R\pi_*\omega_Y(-E_Y)=\pi_*\omega_Y(-E_Y)$
\item For all sufficiently large regular alterations $\pi\colon Y\to X$, $\cO_X$ is locally a summand of $R\pi_*\omega_Y(-E_Y)=\pi_*\omega_Y(-E_Y)$
\end{enumerate}
Here, $E_Y:=s\sum a_kE_k$ where $\cO_Y(-E_Y)=\sJ_k\cO_Y$. By sufficiently large regular alteration we mean any regular alteration $\pi\colon Y\to X$ such that
\[\sJ(X)=\image(\sHom(\pi_*\omega_Y,\cO_X)\otimes_{\cO_X}\pi_*\omega_Y(-E_Y)\to\cO_X).\]
\end{cor}
\begin{proof}
By Theorem \ref{main}, (2)$\implies$(1)$\iff$(3). To see (1)$\iff$(3)$\implies$(2), let $\pi:Y\to X$ a regular log alteration of $(X,\Delta, I)$ and localize so that $\cO_X$ is a summand of $\pi_*\omega_Y$. Then because $X$ has klt type and thus rational singularities we can find maps $p_1,p_2,$ and $i$ such that the following composition is the identity:
\[\begin{tikzcd}
\cO_X\arrow[r]&R\pi_*\cO_Y\arrow[r,"p_1"]&\cO_X\arrow[r,"i"]&\pi_*\omega_Y(-E_Y)\arrow[r,"p_2"]&\cO_X.
\end{tikzcd}\]
\end{proof}

\section{Positive characteristic}

\subsection*{Preliminaries}

Let $R$ be a noetherian ring of characteristic $p>0$. We will denote by $F:R\to R$ the Frobenius morphism defined by $F(r)=r^p$ and denote by $F^e$ the $e$-fold composition of the Frobenius with itself. Given an $R$-module $M$, we denote by $F^e_*M$ the $R$-module where $R$ acts via restriction of scalars along the $e$-th iterate of the Frobenius. We say that $R$ is \emph{$F$-finite} if $F:R\to R$ is a finite map.

Because we do not have resolutions of singularities in characteristic $p$, we are forced to use different methods to study singularities. The Frobenius turns out to be a valuable tool, allowing us to relate singularities of rings and schemes to properties of their Frobenius endomorphism. In particular, we are able to define an analog of the multiplier ideal using such methods.

\begin{definition}
Let $R$ be an $F$-finite, reduced ring. The \emph{test ideal} $\tau(R)$ is the unique smallest ideal $J\subseteq R$, not contained in any minimal prime of $R$, such that $\phi(F^e_*J)\subseteq J$ for every $\phi\in\Hom(F^e_*R,R)$ and every $e>0$.
\end{definition}

This leads to a definition of strongly $F$-regular singularities, the positive characteristic analog of klt singularities.

\begin{definition}
    Let $R$ be an $F$-finite, reduced ring. Then $R$ is \emph{strongly $F$-regular} if $\tau(R)=R$.
\end{definition}

Similarly, we are able to define analogs of multiplier submodules. Suppose $R$ is an $F$-finite, locally equidimensional reduced ring with canonical module $\omega_R$. Consider the $e$-th iterate of the Frobenius $R\to F^e_*R$ and apply the functor $\Hom(-,\omega_R)$. This gives us a map
\[\Hom(F^e_*R,\omega_R)\to\omega_R\]
and using that $\Hom(F^e_*R,\omega_R)\cong F^e_*\omega_R$, we get a map $T^e:F^e_*\omega_R\to\omega_R$ that is dual to the Frobenius.

\begin{definition}[\cite{Smi95},\cite{Bli04},\cite{ST08}]
Let $R$ be an $F$-finite, locally equidimensional ring with canonical module $\omega_R$ and $T:F_*\omega_R\to\omega_R$ dual to the Frobenius. Then the \emph{parameter test submodule} $\tau(\omega_R)$ is the unique smallest submodule $M\subseteq\omega_R$, nonzero at any minimal prime of $R$, such that 
\[T(F_*M)\subseteq M\]
\end{definition}

Ultimately, our result in characteristic $p>2$ is a corollary to the following result of \cite{ES14} combined with the existence of quasi-Gorenstein finite covers.

\begin{prop}[\cite{ES14} Corollary 6.5]\label{test}
If $R$ is an $F$-finite reduced ring of characteristic $p>0$ and $R\subseteq S$ is a finite extension with $S$ reduced, then
\[\tau(R)=\sum_{e\geq0}\sum_\phi\phi(F^e_*\tau(S))\]
where $\phi$ ranges over all elements of $\Hom_R(F^e_*S,R)$.
\end{prop}

We give an overview of a construction of quasi-Gorenstein finite covers in characteriztic not equal to 2 which we expect is well-known to experts. See, for example, the construction in the proof of \cite[Theorem~8.5]{Kawamata:1988}. We include the argument here for the convenience of the reader.

\begin{definition}
    We say that a commutative local ring $R$ is \emph{quasi-Gorenstein} if $R\simeq\omega_R$. More generally, we say a commutative Noetherian ring $R$ is quasi-Gorenstein if all of its local rings are quasi-Gorenstein.
\end{definition}

\begin{lemma}\label{qgor} Let $R$ be a normal commutative Noetherian ring of essentially finite type over a field $k$ of characteristic not equal to 2. Then $\Spec R$ has a quasi-Gorenstein finite cover.
\end{lemma}

\begin{proof}
We first claim that we can find $H\sim -2K_R$ sufficiently general and hence reduced. In characteristic zero, this follows from Bertini's theorem for basepoint-free linear systems, here thinking of $|-2K_R|$ as a linear system on the regular locus of $\Spec R$. In characteristic $p>2$, let $X\to\Spec R$ be the normalized blowup of $\omega_R^{(-2)}$. If $k$ is infinite, then, because $p\neq2$ and $\omega_R^{(-2)}\cO_X\cong\cO_X(-G)$ is very ample, Bertini's theorem tells us that a general divisor $H'\sim-G$ will be normal, hence reduced. Pushing $H'$ down to $H$ on $\Spec R$, we get that $H$ is reduced and $H\sim-2K_R$, as this holds outside a codimension 2 subset of $\Spec R$ and $R$ is normal. If $k$ is not infinite, we can find such an $H$ on $R\otimes_k\overline{\bF_p}$ and thus on $R\otimes_k\bF_{p^e}$ for some $e$. Since we are only looking for a finite cover of $\Spec R$ at the end of the day, we can assume we have such an $H$ by working on $R\otimes_k\bF_{p^e}$.

Given such an $H$ with $H+2K_R=\div f$, let \[S=R\oplus\omega_R=R\oplus\omega_Rt\]
and define multiplication as $(a,b)(x,y)=(ax+byf,ay+bx)$. Since $R$ is regular outside of a set of codimension two, $\omega_R$ will be a line bundle outside of this set and $\omega_R^{(2)}\cong R(-H)$ (via multiplication by $f$) on the smooth locus. Because $R$ is $S_2$, this isomorphism will hold everywhere and so we get an $R$-algebra that is finite as an $R$-module. Furthermore, since $f$ is reduced, $S$ is normal.

We then claim that $S$ is quasi-Gorenstein. To see this, note that $\omega_S\simeq\Hom_R(S,\omega_R)$. However, since $S$ is normal and hence $S_2$, $\omega_S=\Hom_R(S,\omega_R)\simeq S$ and thus $S$ is quasi-Gorenstein.
\end{proof}

\begin{remark}\label{char}
    A priori, there is nothing preventing the existence of quasi-Gorenstein finite covers in characteristic 2. The characteristic restriction is simply an artefact of the proof as we can't guarantee that a general member of $|-2K_R|$ is reduced when $p=2$. 
\end{remark}

\subsection*{Results}

In this section, we present a description of the test ideal that is analagous to the description of the multiplier ideal in Theorem \ref{main}. Note that the restriction on the characteristic is an artefact of the fact that the proof uses Lemma \ref{qgor}

\begin{prop}\label{testmain}
Let $R$ be a Noetherian, $F$-finite reduced ring of characteristic $p>2$ and fix $L_1,\dots, L_n$, algebraic closures of $(R/\fp_i)_{\fp_i}$ for $\fp_i$ the minimal primes of $R$. The test ideal $\tau(R)$ can be realized as
\[\tau(R)=\sum_{R\subset S} \image\left(\Hom_R(\omega_S,R)\otimes_R\tau(\omega_S)\to R\right)\]
where the sum ranges over all finite extensions $R\subset S$ whose total ring of fraction is contained within $\prod L_i$ and $\Hom_R(\omega_S,R)\otimes_R\tau(\omega_S)\to R$ is the evaluation map.
\end{prop}

\begin{proof}
Consider $\phi:\tau(\omega_S)\to R$ and consider $S\to S'$ a quasi-Gorenstein finite cover. By \cite[Proposition~4.4]{BST15} we get that $\phi\circ\Tr_{S'/S}(\tau(\omega_{S'}))=\phi(\tau(\omega_S))$. Thus we can restrict ourselves to only considering quasi-Gorenstein covers in the sum.

Checking equality locally, consider a quasi-Gorenstein cover $S$ with an isomorphism $\omega_S\simeq S$ (and thus $\tau(\omega_S)\simeq\tau(S)$). $R\to F^e_*S$ is finite for all $e$ and so by Proposition \ref{test} and the above observation we have
\[\tau(R)=\sum_{e\geq0}\sum_\phi \phi(F^e_*\tau(S))=\sum_{e\geq0}\sum_\phi \phi(F^e_*\tau(\omega_S))=\sum_{R\to S'}\sum_\phi \phi(\tau(\omega_{S'}))\]
where the final sum again ranges over all $R\hookrightarrow S'$ finite and $\phi\in\Hom_R(\omega_{S'},R)$.

\end{proof}

As a corollary, we get the following characterization of strongly $F$-regular singularities.

\begin{cor}
Let $R$ be a Noetherian, $F$-finite reduced ring of characteristic $p>2$. Then $R$ is strongly $F$-regular if and only if $R$ is a summand of $\tau(\omega_S)$ for any finite extension $R\subset S$ such that $\tau(R)=\image(\Hom_R(\omega_S,R)\otimes_R\tau(\omega_S)\to R)$.
\end{cor}

\begin{proof}
$R$ is strongly $F$-regular if and only if $\tau(R)=R$. By the previous proposition, this implies there is some finite $R\to S$ and a surjection $\tau(\omega_S)\to R$ which must split as a map of $R$-modules. The reverse direction follows from Proposition \ref{testmain} by letting $\pi:\tau(\omega_S)\to R$ be the projection onto the summand, implying $\tau(R)=R$.
\end{proof}

\begin{remark}
    In the case where $R$ is a quasi-Gorenstein domain, the previous corollary follows immediately from \cite[Proposition~4.4]{BST15} which says that $\Tr(\tau(\omega_S))=\tau(\omega_R)\simeq\tau(R)$. See also \cite[Corollary~5.7]{Carvajal-Rojas-Stabler:2023} which uses Cartier modules to achieve the same result. In particular, these results do not assume $p>2$.
\end{remark}

\bibliographystyle{amsalpha}
\bibliography{main}

\end{document}